\documentclass[a4paper,11pt]{article}
\usepackage{amsmath,amsfonts,mathrsfs,amssymb,color}
\usepackage{indentfirst}
\numberwithin{equation}{section}

\newtheorem{theo.}{\quad\, Theorem}[section]
\newtheorem{defi.}{\quad\, Definition}[section]
\newtheorem{lemm.}{\quad\, Lemma}[section]
\newtheorem{coro.}{\quad\, Corollary}[section]

\begin{document}

\title{Positive solutions to semipositone problems   with
mean curvature operator in Minkowski space}
\author{Ao Xiao$^{\ast}$
\\
\small{Department of Mathematics, Northwest Normal University,
Lanzhou\ 730070, P R China}\\
}
\date{} \maketitle
\footnote[0]{E-mail: xiaoao2math@163.com (A. Xiao). $^*$Corresponding author: Ao Xiao.
 }  
%
 \begin{abstract}

We consider the existence of positive solutions for a class of semipositone problems with Minkowski-curvature operators   of the form
$$
  -\left(\frac{u'}{\sqrt{1-(u')^2}}\right)'=f(t,\ u),\ \  t\in(0,\ 1),\ \ \ u'(0)=0=u(1),$$
where $f:  [0,\ 1]\times[0,\ 1)\to\mathbb{R}$ is continuous, $f(t,\ 0)=0$,  and there exists a continuous function $\beta:  [0,\ 1]\to[0,\ +\infty)$
 such that $f(t,\ u)\geq-\beta(t)u$ for $(t,\ u)\in [0,\ 1]\times[0,\ 1).$ Moreover, $f$ satisfies the superlinear  and sublinear   at the origin or $1$ with respect to $\varphi(s)=\frac{s}{\sqrt{1-s^2}}$, respectively.
 The proofs of main
results are based upon  the fixed point index theory in
cones.

\end{abstract}


{\small\bf MR(2020)\ \ 34B15,\ 47H14.}

\baselineskip 20pt

\section{Introduction}
It is well known that
Minkowski
space $\mathcal{L}^{N+1}:=\{(x,\ t):\ x\in\mathbb{R}^N,\ t\in\mathbb{R}\}$ is  endowed with the Lorentzian metric
$\sum\limits_{j=1}^N(dx_j)^2-(dt)^2$.  It is   shown (see [2]) that the study of spacelike submanifolds of codimension one in $\mathcal{L}^{N+1}$
with prescribed mean extrinsic curvature leads to Dirichlet problems of the type
 $$
-\operatorname{div}\left(\frac{\nabla u(x)}{\sqrt{1-|\nabla u(x)|^2}}\right)=f(x,\ u),\quad x\in\Omega,\qquad u=0\ \text{on}\ \partial\Omega,\eqno(1.1)$$
where $\Omega=\{x\in \mathbb{R}^N:  |x|<R\}$ is a ball in $\mathbb{R}^N$.
This type of problems are of interest in differential
geometry and in general relativity. Calabi [4] showed that  the only entire solutions are linear as $N\leq4$. This was later extended to all dimensions by Cheng-Yau [10]. An important universal existence result was proved by Bartnik-Simon [2]. The Dirichlet
problem in a more general spacetime was solved by Gerhardt [16]. Subsequently,
 considerable attention has been devoted to
the existence and multiplicity of solutions for Minkowski-curvature equations and related nonlinear boundary value problems;
see [1, 3, 5--9, 12, 13, 19] and the references therein. For $R=1$, Bereanu et al. [3] proved the existence of at least one
classical positive radial solution of (1.1) when $f$ is nonnegative and is superlinear at the origin and sublinear at $1$ with respect to $\varphi(s)$, that is
$$\lim_{s\to0+}\frac{f(t,\ s)}{s}=\infty,\ \ \ \ \lim_{s\to1-}\sqrt{1-s^2}f(t,\ s)=0,\ \ \ \text{uniformly\ for}\ t\in[0,\ 1].$$
The result in [3] assumes that $f$ is nonnegative. This motivates the study of positive solutions when $f$ is allowed to change sign.

As in [15], Ding and Li showed the existence of  positive solutions of the
  differential equation boundary value problem
  $$-\big(r^{N-1}u'(r)\big)'=r^{N-1} f(r,\ u),\ r\in(r_1,\ r_2),\ \ \ \ u(r_1)=0=u(r_2),$$
where $f:[r_1,\ r_2]\times[0,\ \infty)\to\mathbb{R}$, and  there exists
$\beta\in C\big([r_1,\ r_2],\ [0,\ +\infty)\big)$\ such that
$$f(r,\ u)\geq-\frac{\beta(r)u}{r^{N-1}},\ \ \  (r,\ u)\in[r_1,\ r_2]\times[0,\ \infty).$$
[15] improved and developed  the main results of [18].

  Motivated by the above interesting and important studies, due to the peculiarities of operators $\varphi$,  in this article,
we apply the topological degree to investigate positive   solutions  to the following  semipositone  problems with Minkowski-curvature operators
$$
-\big(\varphi(u')\big)'=f(t,\ u(t)),\ \ t\in(0,\ 1),\ \ \ \ u'(0)=0,\   u(1)=0,\eqno(1.2)
$$
where $f$   is continuous and possibly changes sign.
Let us recall   that
$$
  \varphi(s)=\frac{s}{\sqrt{1-s^2}},\qquad |s|<1.
$$
It is not difficult to show
$
   \varphi'(s)=\frac{1}{(1-s^2)^{3/2}}.$\ Moreover,
  $$ \varphi^{-1}(s)=\frac{s}{\sqrt{1+s^2}},\qquad s\in\mathbb{R},
$$
Clearly,\ $\varphi^{-1}:\ \mathbb{R}\to(-1,\ 1)$\  is odd and strictly increasing in  $\mathbb{R}$. Especially,
$$
 \big|\varphi^{-1}(s)\big|=\varphi^{-1}(|s|)<1,\qquad s\in\mathbb{R}.\eqno(1.3)
$$
 The solution $u$ of   (1.2) is understood in the sense that $u\in C^2(0, 1)\cap C[0, 1]$,
$u$ satisfies (1.2)  and there exists $0<\rho<1$ such that
$$\|u'\|_\infty^2\leq 1-\rho^2,\eqno(1.4)$$
where  $\|u'\|_\infty=\max_{t\in[0,1]}|u'(t)|$.

Throughout  the paper,  we  assume that\par\nopagebreak[4]

(H1)\ $f\in C\big([0,\ 1]\times[0,\ 1),\ \mathbb{R}\big)$,\ $f(t,\ 0)=0$,\ and there exists\
$\beta\in C([0,\ 1],\ [0,\ +\infty))$\ such that
$$
f(t,\ u)\geq-\beta(t)u,\qquad (t,\ u)\in [0,\ 1]\times[0,\ 1).
$$

(H2)\  There exists a constant\  $\sigma>0$\ and a function\
$b_0\in C\big([0,\ 1],\ [0,\ +\infty)\big)$,\ $b_0\not\equiv0$\ on\ $t\in [0,\ 1]$\ such that
$$
f(t,\ u)\geq b_0(t)u,
\qquad (t,\ u)\in [0,\ 1]\times[0,\ \sigma].
$$

(H3)\ There exists a constant\     $0<\gamma<1$\ and a function\
$b_1\in C\big([0,\ 1],\ [0,\ +\infty)\big)$,\ $b_1\not\equiv0$\ on\ $t\in[0,\ 1]$\  such that
$$
f(t,\ u)\leq b_1(t) \varphi(u),
\qquad (t,\ u)\in [0,\ 1]\times[\gamma,\ 1).
$$

(H4)\    There exists a constant\  $\sigma>0$\ and a function\
$b_2\in C\big([0,\ 1],\ [0,\ +\infty)\big)$,\ $b_2\not\equiv0$\ on\ $t\in [0,\ 1]$\  such that
$$
f(t,\ u)\leq b_2(t)u,
\qquad (t,\ u)\in[0,\ 1]\times[0,\ \sigma].
$$

(H5)\ There exists a constant\     $0<\gamma<1$\ and a function\
$b_3\in C\big([0,\ 1],\ [0,\ +\infty)\big)$,\ $b_3\not\equiv0$\ on\ $t\in[0,\ 1]$\  such that
$$
f(t,\ u)\geq b_3(t) \varphi(u),
\qquad (t,\ u)\in [0,\ 1]\times[\gamma,\ 1).
$$

\vskip 2mm

 Set\ $i=0,\ 1,\ 2,\ 3$.  Let   $b_i$  satisfy the above assumptions.\ It is well known (see for example \mbox{[11, 20]})\ that the following linear eigenvalue problem
$$
-u''=\lambda b_i(t)u,\qquad
u'(0)=0,\ \ u(1)=0
$$
possesses infinitely many positive eigenvalues  $0<\lambda_1(b_i)<\lambda_2(b_i)<\cdots<\lambda_n(b_i)\to\infty,$ all
of which are simple. The   unit eigenfunction
 $\phi_n(b_i)$ (that is $\|\phi_n(b_i)\|_\infty=\max_{t\in[0,1]}|\phi_n(b_i)(t)|=1$) corresponding to  $\lambda_n(b_i)$  has exactly  $n-1$
 simple zeros in $(0,\ 1)$.


\vskip.3cm

The following two theorems are some main results of this paper.

\noindent{\bf Theorem 1.1.}\ \ Assume that (H1)--(H3) hold. Let
  $\lambda_1(b_0)<\rho^3$ and  $\lambda_1(b_1)>\frac{1}{\rho}$.  There exists some constant  $C_0>0$ such that
$$C_0\int_0^1\phi_1(b_1)(t)dt<\Big(\lambda_1(b_1)-\frac{1}{\rho}\Big)\delta^2\int_0^1 b_1(t) G_\beta^2(t,\ t)dt, \eqno(1.5)$$
where $G_\beta$ and $\delta$ are defined in (2.5) and Lemma 2.2, respectively. Then (1.2) has at least one positive solution.

\vskip 3mm

\noindent{\bf Theorem 1.2.}\ \ Assume that (H1) and (H4)--(H5) hold. Let
  $\lambda_1(b_2)>1$ and $\lambda_1(b_3)<\rho^3$. There exists some constant  $C_1>0$ such that
$$C_1\int_0^1\phi_1(b_3)(t)dt<\big(\rho^3-\lambda_1(b_3)\big)\delta^2\int_0^1 b_3(t) G_\beta^2(t,\ t)dt, \eqno(1.6)$$
Then (1.2) has at least one positive solution.

The rest part of this paper is organized as follows. Section 2 is stated some preliminary
results. Section 3 is devoted to showing the proofs of main results in this paper and giving an application about Theorem 1.1.

\section{Preliminaries}

Denote by $E=C([0,\ 1])$ the Banach space under the norm $\|u\|=\max_{t\in [0, 1]}|u(t)|$.

Let us define the auxiliary function
$$
  h(s)=
  \begin{cases}
  (1-s^2)^{3/2},&|s|\leq1,\\
  0,&|s|>1.
  \end{cases}\eqno(2.1)
$$
So
$$
  0\leq h(s)\leq1,\qquad s\in\mathbb{R}.
$$
Consider the auxiliary  semipositone problem
$$
-u''(t)=f\big(t,\ u(t)\big)h(u'(t)),\ \  t\in(0,\ 1),\ \ \ u'(0)=0,\  u(1)=0.\eqno(2.2)$$
 The following will show  that solving  the positive solution of (2.2) is equivalent to proving  the existence of
the   positive solution   of (1.2).

\vskip.2cm

\noindent{\bf Lemma 2.1.}\ \ Assume that (H1) holds.  Then  $u$ is a
solution of (1.2) if and only if it is a
solution of (2.2).

\vskip.2cm

\noindent{\bf Proof.}\ \ Let $u$ be a solution of (1.2). Then   $\|u'\|_\infty<1$. Since
$$
  \big(\varphi(u')\big)'=\frac{u''}{\big(1-(u')^2\big)^{3/2}},$$
it follows from (1.2) that
$$
  -u''=f(t,u)\big(1-(u')^2\big)^{3/2}=f(t,\ u)h(u'),$$
 then $u$ is a solution of (2.2).

For the converse, if $u$ is a solution of (2.2). We only need to prove
$$
  \|u'\|_\infty<1,\eqno(2.3)
$$
where  $\|u'\|_\infty=\max_{t\in[0,1]}|u'(t)|$.  Suppose on the contrary that (2.3) does not hold,
Then, combining $u'(0)=0$  with the continuity of  $u'$, the set
$$
  \big\{t:\ t\in[0,\ 1],\  |u'(t)|=1\big\}$$
is nonempty. Denote its minimum element as   $t_*$.  Then
$$
  |u'(t)|<1,\ \ \ 0\leq t<t_*,\ \ \ \  |u'(t_*)|=1.
$$
  $h(u'(t))=\big(1-(u'(t))^2\big)^{3/2}$ for $[0,\ t_*)$. Thus,
$$
  -\big(\varphi(u')\big)'=f(t,\ u),\ \ \ \ t\in[0,\ t_*).\eqno(2.4)
$$
Integrating both sides of the above equation (2.4) over $(0,\ t)$ for  $0\leq t<t_*$, we get
$$
  u'(t)=-\varphi^{-1}\left(\int_0^t f(s,\ u(s))ds\right).
$$
Set
$$
  M_*:=\int_0^{t_*}|f(s,\ u(s))|ds,
$$
then\ $M_*<\infty$.\ It follows from (1.3) and the monotonicity of $\varphi^{-1}$
that
\begin{align*}
 |u'(t)|
  &=\left|
  \varphi^{-1}\left(\int_0^t f(s,\ u(s))ds\right)
  \right|\\
  &=\varphi^{-1}\left(\left|\int_0^t f(s,\ u(s))ds\right|\right)\\
  &\leq \varphi^{-1}\left(\int_0^t|f(s,u(s))|ds\right)
  \leq \varphi^{-1}(M_*)<1.
\end{align*}
When\ $t\to t_*^-$, we obtain
$$
  1=|u'(t_*)|\leq\varphi^{-1}(M_*)<1,
$$
which is a contradiction.\ Therefore,  (2.3) is valid, and hence $h(u'(t))=\big(1-(u'(t))^2\big)^{3/2}$, combining
this with the equation of (2.2), it follows that $-\big(\varphi(u')\big)'=f(t,\ u)$. Consequently,  $u$ is also a solution of  (1.2).
{\hfill$\square$}

\vskip 2mm

Let $\eta,\ \zeta$  be the unique solution of the linear boundary value problem
$$
\begin{cases}
-\eta''+\beta(t)\eta=0,\ \ \ t\in(0,\ 1),\\
\eta(0)=1,\quad \eta'(0)=0
\end{cases}
$$
and
$$
\begin{cases}
-\zeta''+\beta(t)\zeta=0,\ \ \ t\in(0,\ 1),\\
\zeta(1)=0,\quad \zeta'(1)=-1
\end{cases}
$$
respectively.  It follows from [18] that
$$
  \eta(t)>0,\quad \eta'(t)\geq0,\quad t\in[0,\ 1],\ \
  \qquad
  \zeta(t)>0,\quad \zeta'(t)<0,\quad 0\leq t<1.
$$
Set $W(t)=\eta'(t)\zeta(t)-\eta(t)\zeta'(t)$.\ It is not difficult to verify that
$W'\equiv0$ for $t\in[0,\ 1]$,  moreover,
$W(1)=-\eta(1)\zeta'(1)=\eta(1)>0$.\ We now define the Green function $G_\beta:[0,\ 1]\times[0,\ 1]\to[0,\ \infty)$,
$$
G_\beta(t,\ s)=\frac{1}{\eta(1)}
\begin{cases}
\eta(t)\zeta(s),&0\le t\le s\le 1,\\
\eta(s)\zeta(t),&0\le s\le t\le 1.
\end{cases}\eqno(2.5)
$$
It follows from [18] that $G_\beta(t,\ s)$ has some important
properties.

\vskip 3mm

\noindent{\bf Lemma 2.2.}\ \ ([18])\quad $G_\beta(t,\ s)$ has the following
properties.
\vskip 2mm

(1)\ $G_\beta(t,\ s)>0$,\ \ \  $t,\ s\in(0,\ 1)$;

(2)\  $G_\beta(t,\ s)\leq G_\beta(s,\ s)$,\ \ \ $t,\ s\in[0,\ 1]$;

(3)\  $G_\beta(t,\ s)\geq \delta G_\beta(t,\ t)G_\beta(s,\ s),\ \ \ t,\ s\in[0,\ 1]$, where $\delta=\frac{1}{\zeta(0)}$.

\vskip 3mm

By virtue of  Lemma 2.2, it is easy to obtain the following key lemma, the  proof is omitted.

\noindent{\bf Lemma 2.3.}\ \ ([18])\quad  Let\ $g:\ [0,\ 1]\to[0,\ +\infty)$ is continuous. The linear mixed boundary value problem
$$
-u''(t)+\beta(t)u(t)=g(t), \ \ t\in(0,\ 1),\ \ \
u'(0)=0,\ \ u(1)=0
$$
 has the unique solution $u$,
 $$u(t)=\int_0^1G_\beta(t,\ s)g(s)ds,$$
 and  $u(t)\geq \delta G_\beta(t,\ t)\|u\|$ for $t\in[0,\ 1]$.

  \vskip 3mm

Define the linear operator $S_\beta$ by
$$
  (S_\beta g)(t):=\int_0^1G_\beta(t,\ s)g(s)\,ds,\qquad g\in E.
$$

Define   $K$ a cone in $E$ by
$$
K=\big\{
u\in E:\
u(t)\geq\delta G_\beta(t,\ t)\|u\|,\   t\in[0,\ 1]
\big\}.
$$
For $u\in K$,\ set\ $F(u)(t)=f(t,\ u(t))h(u'(t))+\beta(t)u(t)$.\ Let the map $A$ be defined by
$$
  (Au)(t)=(S_\beta F(u))(t)=\int_0^1G_\beta(t,\ s)\Big(f(s,\ u(s))h(u'(s))+\beta(s)u(s)\Big)\,ds.
$$
Note that, from (H1) and $0\leq h(s)\leq1$, we have   $f(t,\ u)h(u')+\beta(t)u\geq0$.
Then it is not difficult to show that $A:\ K\to K$ is a completely continuous operator, and solving
positive solution of problems (1.2) is equivalent to proving the nonzero fixed point of $A$.

\vskip.3cm

Now we state  some conclusions on the fixed point index theory in cones [14, 17].

\noindent{\bf Lemma 2.4.}\ \ Let  $K $  be a cone in   Banach  space   $E $,
  $\Omega\subset E$ be a bounded open subset of $E$ with $0\in\Omega$, and
 $A:\ K\cap\overline\Omega\to K$ be  a
completely continuous mapping. If $\mu Au\ne u$ for every  $u\in K\cap\partial\Omega$ and $0<\mu\le1$.
Then $i(A,\ K\cap\Omega,\ K)=1$.

\vskip.3cm

\noindent{\bf Lemma 2.5.}\ \ Let  $K $  be a cone in   Banach  space   $E $,
  $\Omega\subset E$ be a bounded open subset of $E$ with $0\in\Omega$, and
 $A:\ K\cap\overline\Omega\to K$ be a
completely continuous mapping. If there exists $e\in K\setminus\{0\}$ such that  $u-Au\ne\tau e$ for every  $u\in K\cap\partial\Omega$ and $\tau\ge0$.
Then $i(A,\ K\cap\Omega,\ K)=0$.

\section{Proofs of main results}

\noindent{\bf Proof of Theorem 1.1.}\ \
Choose $0<r<\sqrt{1-\rho^2}$ such that $r<\sigma$. Set
$$
\Omega_1=\left\{u\in E:\ \|u\|<r\right\}.
$$
Now let us prove
$$
u-Au\ne\tau \phi_1(b_0),\qquad
u\in K\cap\partial\Omega_1,\quad\tau\geq0.\eqno(3.1)
$$
Suppose on the contrary that there exists  $u_0\in K\cap\partial\Omega_1$ and $\tau_0\geq0$ such that
$$
u_0=Au_0+\tau_0 \phi_1(b_0).\eqno(3.2)
$$
According to  (1.4) and (2.1), we have
$$h(u_0')=\big(1-(u_0')^2\big)^{3/2}\geq\rho^3.$$
Moreover, it follows from (3.2) that
$$-u_0''(t)+\beta(t)u_0=f(t,\ u_0)h(u_0')+\beta(t)u_0+\tau_0\lambda_1(b_0)b_0(t)\phi_1(b_0)(t)+\tau_0\beta(t)\phi_1(b_0)(t),$$
then, by virtue of  (3.2) and  (H2), we are driven to
$$
\aligned
&-u_0''(t)+\beta(t)u_0\\
\geq&h(u_0')b_0(t)u_0+\beta(t)u_0+\tau_0\lambda_1(b_0)b_0(t)\phi_1(b_0)(t)+\tau_0\beta(t)\phi_1(b_0)(t)\\
\geq& \rho^3b_0(t)u_0+\beta(t)u_0,\\
\endaligned
$$
and hence
$$
\aligned
&\int_0^1\big(-u_0''(t)+\beta(t)u_0\big)\phi_1(b_0)(t)dt\\
\geq&\int_0^1\big(\rho^3b_0(t)u_0+\beta(t)u_0\big)\phi_1(b_0)(t)dt.\\
\endaligned \eqno(3.3)
$$
On the other hand,
$$
\aligned
&\int_0^1\big(-u_0''(t)+\beta(t)u_0\big)\phi_1(b_0)(t)dt\\
=&\int_0^1\big(-\phi_1(b_0)''(t)+\beta(t)\phi_1(b_0)\big)u_0(t)dt\\
=&\lambda_1(b_0)\int_0^1 b_0(t)u_0(t)\phi_1(b_0)(t)dt+\int_0^1 \beta(t)u_0(t)\phi_1(b_0)(t)dt.\\
\endaligned \eqno(3.4)
$$
Therefore, from (3.3) and (3.4), we obtain
$$\lambda_1(b_0)\int_0^1 b_0(t)u_0(t)\phi_1(b_0)(t)dt\geq\rho^3\int_0^1  b_0(t)u_0(t)\phi_1(b_0)(t)dt.$$
it follows from Lemma 2.2 that
$$u_0(t)\geq \delta G_\beta(t,\ t)\|u_0\|_\infty,\ \ \ \phi_1(b_0)(t)>0\ \text{in}\ t\in(0,\ 1), $$
and hence
$$\int_0^1  b_0(t)u_0(t)\phi_1(b_0)(t)dt>0.$$
Thus,  $\lambda_1(b_0)\geq \rho^3$, which contradicts with $\lambda_1(b_0)<\rho^3$.
Therefore, (3.1) is valid, and so  it follows from Lemma 2.5 that
$$
i(A,\ K\cap\Omega_1,\ K)=0.\eqno(3.5)
$$

\vskip 3mm

On the other hand,
set
$$R_0:=\frac{C_0\int_0^1\phi_1(b_1)(t)dt}{\Big(\lambda_1(b_1)-\frac{1}{\rho}\Big)\delta^2\int_0^1 b_1(t) G_\beta^2(t,\ t)dt},$$
from   (1.5), we get  $0<R_0<1$. Now choose   $\max\{R_0,\ \gamma\}<R<1$,\
the following will show
$$\mu Au\neq u,\ \ \ u\in K\cap\partial \Omega_2,\ \ \ 0<\mu\leq1,\eqno(3.6)$$
where $\Omega_2=\left\{u\in E:\ \|u\|<R\right\}.$

Suppose on the contrary that there exists  $u_1\in K\cap\partial \Omega_2,\ 0<\mu_1\leq1$ such that $\mu_1 Au_1=u_1$. So
$$-u_1''(t)+\beta(t)u_1=\mu_1\big[f(t,\ u_1)h(u_1')+\beta(t)u_1\big],\ \ t\in(0,\ 1),\ \ \ u_1'(0)=0=u_1(1).\eqno(3.7)$$
Set
$$C_0=\max\big\{|f(t,\ \xi)-b_1(t)\varphi(\xi)|:\ t\in[0,\ 1],\ \xi\in[0,\ \gamma]\big\},$$
From (H3), we have
$$f(t,\ u_1)\leq  b_1(t)\varphi(u_1)+C_0,\ \ 0\leq u_1<1,$$
and hence
$$f(t,\ u_1)h(u_1')\leq  b_1(t)\varphi(u_1)+C_0,\ \ 0\leq u_1<1,$$
where we have used  $0\leq h(u_1')\leq1$. Then  it follows from
  (3.7) that
$$-u_1''(t)+\beta(t)u_1\leq b_1(t)\varphi(u_1)+C_0+\beta(t)u_1.\eqno(3.8) $$
Multiplying (3.8) by $\phi_1(b_1)$ and integrating  from $0$ to $1$, it follows that
$$
\aligned
&\int_0^1\big[-u_1''(t)+\beta(t)u_1\big]\phi_1(b_1)(t)dt\\
\leq&\int_0^1\Big[\big(b_1(t)\varphi(u_1)+C_0\big)\phi_1(b_1)(t)+\beta(t)\phi_1(b_1)(t)u_1(t)\Big]dt\\
=& \int_0^1\big(b_1(t)\varphi(u_1)+C_0\big)\phi_1(b_1)(t)dt+ \int_0^1\beta(t)\phi_1(b_1)(t)u_1(t)dt\\
\leq& \int_0^1\Big(b_1(t)\frac{u_1(t)}{\rho}+C_0\Big)\phi_1(b_1)(t)dt+ \int_0^1\beta(t)\phi_1(b_1)(t)u_1(t)dt.\\
\endaligned\eqno(3.9)
$$
In addition,  by using integration by parts,
  we get
$$
\aligned
&\int_0^1\big[-u_1''(t)+\beta(t)u_1\big]\phi_1(b_1)(t)dt\\
=&\int_0^1\big[-\phi_1(b_1)''(t)+\beta(t)\phi_1(b_1)(t)\big]u_1(t)dt\\
=&\lambda_1(b_1)\int_0^1 b_1(t)u_1(t)\phi_1(b_1)(t)dt+\int_0^1 \beta(t)u_1(t)\phi_1(b_1)(t)dt.\\
\endaligned\eqno(3.10)
$$
Therefore, from (3.9) and (3.10), we have
$$\int_0^1 b_1(t)u_1(t)\phi_1(b_1)(t)dt\leq\frac{C_0}{\lambda_1(b_1)-\frac{1}{\rho}} \int_0^1\phi_1(b_1)(t)dt.\eqno(3.11)$$
Combining this with Lemma 2.2, we get
$$u_1(t)\geq \delta G_\beta(t,\ t)\|u_1\|,\ \ \ \phi_1(b_1)(t)\geq \delta G_\beta(t,\ t)\|\phi_1(b_1)\|=\delta G_\beta(t,\ t).$$
Thus,
$$\int_0^1 b_1(t)u_1(t)\phi_1(b_1)(t)dt\geq \delta^2 \|u_1\|\int_0^1 b_1(t) G_\beta^2(t,\ t)dt,$$
together this with (3.11), then
$$\|u_1\|\leq\frac{C_0\int_0^1\phi_1(b_1)(t)dt}{\Big(\lambda_1(b_1)-\frac{1}{\rho}\Big)\delta^2\int_0^1 b_1(t) G_\beta^2(t,\ t)dt}=R_0. $$
But we know that  $\|u_1\|=R>R_0$, which is a contradiction. Then  (3.6) is valid. It follows from Lemma 2.4 that
$$
i(A,\ K\cap\Omega_2,\ K)=1.\eqno(3.12)
$$
Therefore, based on (3.5) and (3.12), we can obtain
$$
i(A,\ K\cap\Omega_2\setminus\overline{\Omega_1},\ K)=i(A,\ K\cap\Omega_2,\ K)-i(A,\ K\cap\Omega_1,\ K)=1.
$$
Then  $A$  has a fixed point $\bar{u}$, which is a positive solution  of (1.2).
{\hfill$\square$}

\vskip 4mm

\noindent{\bf Proof of Theorem 1.2.}\ \ Choose $0<r<\sigma$, we need to show that
$$\mu Au\neq u,\ \ \ u\in K\cap\partial \Omega_1,\ \ \ 0<\mu\leq1.\eqno(3.13)$$
Suppose on the contrary that there exist  $u_2\in K\cap\partial \Omega_1$  and  $0<\mu_2\leq1$ such that
$\mu_2 Au_2=u_2$. Then
$$-u_2''(t)+\beta(t)u_2(t)=\mu_2\big[f(t,\ u_2)h(u_2')+\beta(t)u_2(t)\big],\ \ t\in(0,\ 1),\ \ \ u_2'(0)=0=u_2(1).\eqno(3.14)$$
Note that $u_2\in K\cap\partial \Omega_1$, and hence $u_2(t)\leq\|u_2\|=r<\sigma$, then it follows from   (H4) that
$$f(t,\ u_2)\leq  b_2(t)u_2,\ \ 0\leq u_2\leq r.$$
So
$$f(t,\ u_2)h(u_2')\leq  b_2(t)u_2,\ \ 0\leq u_2\leq r,$$
 Combining this with $0<\mu_2\leq1$ and (3.14), we have
$$-u_2''(t)+\beta(t)u_2\leq b_2(t)u_2+\beta(t)u_2.\eqno(3.15) $$
Multiplying (3.15) by $\phi_1(b_2)$ and integrating  from $0$ to $1$, we get
$$
\aligned
&\int_0^1\big[-u_2''(t)+\beta(t)u_2(t)\big]\phi_1(b_2)(t)dt\\
\leq&\int_0^1\Big[b_2(t)u_2(t)\phi_1(b_2)(t)+\beta(t)\phi_1(b_2)(t)u_2(t)\Big]dt\\
\endaligned\eqno(3.16)
$$
Furthermore, by using integration by parts, we have
$$
\aligned
&\int_0^1\big[-u_2''(t)+\beta(t)u_2\big]\phi_1(b_2)(t)dt\\
=&\int_0^1\big[-\phi_1(b_2)''(t)+\beta(t)\phi_1(b_2)(t)\big]u_2(t)dt\\
=&\lambda_1(b_2)\int_0^1 b_2(t)u_2(t)\phi_1(b_2)(t)dt+\int_0^1 \beta(t)u_2(t)\phi_1(b_2)(t)dt.\\
\endaligned\eqno(3.17)
$$
Therefore, from (3.16) and (3.17), we obtain
$$\lambda_1(b_2)\int_0^1 b_2(t)u_2(t)\phi_1(b_2)(t)dt\leq\int_0^1  b_2(t)u_2\phi_1(b_2)(t)dt.$$
it follows from Lemma 2.2 that
$$u_2(t)\geq \delta G_\beta(t,\ t)\|u_2\|_\infty,$$
and hence
$$\int_0^1  b_2(t)u_2(t)\phi_1(b_2)(t)dt>0.$$
Consequently,  $\lambda_1(b_2)\leq 1$, which contradicts with $\lambda_1(b_2)>1$, and so
    (3.13) holds. Then, from   Lemma  2.4, we have
$$
i(A,\ K\cap\Omega_1,\ K)=1.\eqno(3.18)
$$

\vskip 3mm

Set
$$
C_1:=\max\big\{|f(t,\ u)-b_3(t)\varphi(u)|:\ t\in[0,\ 1],\ u\in[0,\ \gamma]\big\},
$$
then,  from (H5), we get
$$
f(t,\ u)\geq b_3(t)\varphi(u)-C_1,
\qquad (t,\ u)\in[0,\ 1]\times[0,\ 1).\eqno(3.19)
$$
Denote
$$
\overline{R_0}:=\frac{C_1\int_0^1\phi_1(b_3)(t)dt}
{\big(\rho^3-\lambda_1(b_3)\big)\delta^2
\int_0^1b_3(t)G_\beta^2(t,\ t)dt}.
$$
By virtue of  (1.6), we have $\overline{R_0}<1$. Thus,  we may choose  $R$ such that
$$
\max\{\overline{R_0},\ \gamma\}<R<1.
$$
We will need to show that
$$
u-Au\ne\tau \phi_1(b_3),\qquad
u\in K\cap\partial\Omega_2,\quad\tau\geq0.\eqno(3.20)
$$
Suppose on the contrary that there exist   $u_3\in K\cap\partial\Omega_2$ and   $\tau_3\geq0$ such that
$$
u_3=Au_3+\tau_3 \phi_1(b_3).\eqno(3.21)
$$
According to (1.4), we have
$$
h(u_3')=\big(1-(u_3')^2\big)^{3/2}\geq\rho^3.
$$
Since $0\leq u_3(t)\leq\|u_3\|_\infty=R<1$, then
$\varphi(u_3)\geq u_3$. From (3.19) and  $0\leq h(u_3')\leq1$, we get
$$
\begin{aligned}
f(t,\ u_3)h(u_3')
&\geq b_3(t)\varphi(u_3)h(u_3')-C_1h(u_3')\\
&\geq\rho^3b_3(t)\varphi(u_3)-C_1\\
&\geq\rho^3b_3(t)u_3(t)-C_1.
\end{aligned}\eqno(3.22)
$$
Combining this with $-\phi_1(b_3)''(t)=\lambda_1(b_3)b_3(t)\phi_1(b_3)$, then, it follows from (3.21) that
$$
\begin{aligned}
-u_3''(t)+\beta(t)u_3(t)&=f(t,\ u_3)h(u_3')+\beta(t)u_3(t)\\
&\quad+\tau_3\lambda_1(b_3)b_3(t)\phi_1(b_3)(t)
+\tau_3\beta(t)\phi_1(b_3)(t).
\end{aligned}
\eqno(3.23)
$$
 Together (3.22) with (3.23), we have
$$
-u_3''(t)+\beta(t)u_3(t)
\geq \rho^3b_3(t)u_3(t)-C_1+\beta(t)u_3(t)
+\tau_3\lambda_1(b_3)b_3(t)\phi_1(b_3)(t)
+\tau_3\beta(t)\phi_1(b_3)(t).
$$
Multiplying this inequality by $\phi_1(b_3)(t)$ and integrating  from  $0$ to $1$, we have
$$
\begin{aligned}
&\int_0^1\big[-u_3''(t)+\beta(t)u_3(t)\big]\phi_1(b_3)(t)dt\\
\geq&\ \rho^3\int_0^1b_3(t)u_3(t)\phi_1(b_3)(t)dt
-C_1\int_0^1\phi_1(b_3)(t)dt
+\int_0^1\beta(t)u_3(t)\phi_1(b_3)(t)dt.
\end{aligned}
$$
Additionally,
$$
\aligned
&\int_0^1\big[-u_3''(t)+\beta(t)u_3(t)\big]\phi_1(b_3)(t)dt\\
=&\int_0^1\big[-\phi_1(b_3)''(t)+\beta(t)\phi_1(b_3)\big]u_3(t)dt\\
=&\lambda_1(b_3)\int_0^1 b_3(t)u_3(t)\phi_1(b_3)(t)dt+\int_0^1 \beta(t)u_3(t)\phi_1(b_3)(t)dt.\\
\endaligned
$$
Thus,
$$
\big(\rho^3-\lambda_1(b_3)\big)
\int_0^1b_3(t)u_3(t)\phi_1(b_3)(t)dt
\leq C_1\int_0^1\phi_1(b_3)(t)dt.\eqno(3.24)
$$
Since $u_3,\ \phi_1(b_3)\in K$ and $\|u_3\|=R$, then
$$
\int_0^1b_3(t)u_3(t)\phi_1(b_3)(t)dt
\geq\delta^2R\int_0^1b_3(t)G_\beta^2(t,\ t)dt.\eqno(3.25)
$$
via (3.24) and (3.25), it deduces that
$$
R\leq\frac{C_1\int_0^1\phi_1(b_3)(t)dt}{\delta^2\big(\rho^3-\lambda_1(b_3)\big)\int_0^1b_3(t)G_\beta^2(t,\ t)dt}=\overline{R_0},
$$
which contradicts with  $R>\overline{R_0}$, and hence (3.20) is valid. From Lemma  2.5, we have
$$
i(A,\ K\cap\Omega_2,\ K)=0.\eqno(3.26)
$$
Therefore, it deduces from (3.18) and  (3.26), we infer that
$$
i\big(A,\ K\cap(\Omega_2\setminus\overline{\Omega_1}),\ K\big)=i(A,\ K\cap\Omega_2,\ K)-i(A,\ K\cap\Omega_1,\ K)=-1.
$$
Then  $A$  has a fixed point $\bar{u}$, which is a positive solution  of (1.2).
{\hfill$\square$}

\vskip 3mm

We conclude with an example, which illustrates the application of Theorem 1.1.

\noindent{\bf Example 3.1.}\quad
Consider the problem
$$
-\left(\frac{u'}{\sqrt{1-(u')^2}}\right)'=f(t,\ u),
\qquad t\in(0,\ 1),\eqno(3.27)
$$
subject to the mixed boundary conditions $
u'(0)=0,\ u(1)=0,$
where
$$
f(t,\ u)=\left(1+\frac{t}{10}\right)m(u)u,
$$
and
$$
m(u)=
\begin{cases}
4, & 0\leq u\leq\dfrac{1}{20},\\[2mm]
\dfrac{81}{10}-82u,
& \dfrac{1}{20}<u<\dfrac{1}{10},\\[2mm]
-\dfrac{1}{10}, & \dfrac{1}{10}\leq u<1.
\end{cases}
$$
It is  not difficult to verify that  $m$ is continuous and sign-changing.
Consequently, the nonlinearity $f$ is a sign-changing.
Set
$$
\rho=\frac{9}{10},\qquad \beta(t)\equiv1,\qquad
\sigma=\frac{1}{20},\qquad \gamma=\frac{1}{10},
$$
and
$$
b_0(t)\equiv4,\qquad b_1(t)\equiv\frac{1}{10}.
$$
Since $m(u)\geq-\frac{1}{10}$ and $0\leq t\leq1$, we have
$$
f(t,u)\geq-\frac{11}{100}u
\geq-u=-\beta(t)u,
$$
which implies that (H1) holds. For $0\leq u\leq\frac{1}{20}$,
$$
f(t,u)=4\left(1+\frac{t}{10}\right)u
\geq4u=b_0(t)u,
$$
and hence (H2) holds. Moreover,
$$
\lambda_1(b_0)=\frac{\pi^2}{16}
<\left(\frac{9}{10}\right)^3=\rho^3.
$$

For $\frac{1}{10}\leq u<1$, we have
$$
f(t,u)=-\frac{1}{10}\left(1+\frac{t}{10}\right)u
\leq\frac{1}{10}\varphi(u)=b_1(t)\varphi(u).
$$
Therefore, (H3) holds. Furthermore,
$$
\lambda_1(b_1)=\frac{\pi^2}{4(1/10)}
=\frac{5\pi^2}{2}>\frac{10}{9}=\frac{1}{\rho}.
$$

For $\beta(t)\equiv1$, then the Green function
$$
G_\beta(t,t)=\frac{\cosh t\,\sinh(1-t)}{\cosh1},
\qquad \delta=\frac{1}{\sinh1}.
$$
Since $b_1(t)\equiv\frac{1}{10}$, then
$$
\phi_1(b_1)(t)=\cos\frac{\pi t}{2}.
$$
Consequently,
$$
\int_0^1\phi_1(b_1)(t)\,dt=\frac{2}{\pi},
\qquad
\int_0^1G_\beta^2(t,t)\,dt\approx0.187709.
$$

Let
$$
C_0=\max_{\substack{0\leq t\leq1\\0\leq u\leq1/10}}
\left|f(t,u)-\frac{1}{10}\varphi(u)\right|.
$$
Since $|m(u)u|\leq\frac{1}{5}$ for $0\leq u\leq\frac{1}{10}$, it follows that
$$
C_0\leq\frac{11}{50}+\frac{1}{100\sqrt{0.99}}<0.230051.
$$
Hence,
$$
C_0\int_0^1\phi_1(b_1)(t)\,dt<0.146455.
$$
On the other hand,
$$
\left(\lambda_1(b_1)-\frac{1}{\rho}\right)
\delta^2\int_0^1b_1(t)G_\beta^2(t,t)\,dt
\approx0.320250.
$$
Thus, the condition (1.5) is satisfied. It follows from Theorem 1.1 that there exists at least one positive solution for
 the equation (3.27) with the mixed boundary conditions.

\vskip 3mm


\noindent{\bf Availability of data and materials}\\
{\small Data sharing not applicable to this paper as no datasets were generated or analyzed during the current study.}

\noindent{\bf
Funding information}\\
{\small
This work is supported by the
NSFC (No. 12361040).}


 \vskip 3mm

\centerline {\bf REFERENCES}
\begin{list}{}{\setlength{\leftmargin}{2.5em}
\setlength{\labelwidth}{2em}\setlength{\labelsep}{0.5em}
\setlength{\itemindent}{0pt}\setlength{\listparindent}{0pt}}

\item[{[1]}] A. Azzollini, Ground state solution for a problem with mean curvature operator in Minkowski space, J. Funct. Anal., 266 (2014): 2086--2095.

\item[{[2]}] R. Bartnik, L. Simon, Spacelike hypersurfaces with prescribed boundary values and mean curvature, Commun. Math.
Phys., 87(1) (1982/83): 131--152.

\item[{[3]}] C. Bereanu, P. Jebelean, P. J. Torres, Positive radial solutions for Dirichlet problems with mean curvature operators in
Minkowski space, J. Funct. Anal., 264 (2013): 270--287.

\item[{[4]}] E. Calabi, Examples of Bernstein problems for some nonlinear equations, in: Global Analysis (Proc. Sympos. Pure Math., Vol. XV, Berkeley, Calif., 1968), Amer. Math. Soc., Providence, R.I., 1970, pp. 223--230.

\item[{[5]}] T. Chen, Bifurcation from interval and positive solutions of Minkowski-curvature on unbounded domain, J. Math. Anal. Appl., 548 (2025): 129422.

\item[{[6]}] T. Chen, C. S. Goodrich, Global structure of multiple positive solutions for the $k$-th mean curvature in Minkowski space, J. Fixed Point Theory Appl., 27 (2025): Article 48.

\item[{[7]}] T. Chen, H. Wu,
Existence and nonexistence of solutions of Minkowski-curvature problems in exterior domains, Q. J. Math., 75 (2024): 735--748.

\item[{[8]}] T. Chen, Y. Zhao, Existence of solutions for systems of Minkowski-curvature Neumann problems, Rocky Mountain J. Math., 53(5) (2023): 1431--1444.

\item[{[9]}] T. Chen, R. Ma, Y. Liang, Multiple positive solutions of second-order nonlinear difference equations with discrete singular $\phi$-Laplacian,
J. Difference Equ. Appl., 25(1) (2019): 38--55.

\item[{[10]}] S. Y. Cheng, S. T. Yau, Maximal space-like hypersurfaces in the Lorentz-Minkowski spaces, Ann. of Math., 104 (1976): 407--419.

\item[{[11]}] E. A. Coddington, N. Levinson, Theory of Ordinary Differential Equations,
McGraw-Hill, New York, 1955.

\item[{[12]}] G. Dai, J. Yao, F. Li, Spectrum and bifurcation for semilinear elliptic problems in $\mathbb{R}^N$, J. Differ. Equ., 263 (2017): 5939--5967.

\item[{[13]}] G. Dai, J. Wang, Nodal solutions to problem with mean curvature operator in Minkowski space, Differential Integral Equations,
30(5--6) (2017): 463--480.

\item[{[14]}] K. Deimling, Nonlinear Functional Analysis, Springer-Verlag, Berlin, 1985.

\item[{[15]}]
Y. Ding, Y. Li, Positive radial solutions for elliptic equations with
sign-changing nonlinear terms in an annulus,
Complex Var. Elliptic Equ., 67 (2022): 1229--1243.

\item[{[16]}] C. Gerhardt, H-surfaces in Lorentzian manifolds, Commun. Math. Phys., 89 (1983): \mbox{523--553}.

\item[{[17]}] D. Guo, V. Lakshmikantham, Nonlinear Problems in Abstract Cones, Academic Press, Boston, 1988.

\item[{[18]}] Y. Li, On the existence and nonexistence of positive
solutions for nonlinear Sturm-Liouville boundary
value problems, J. Math. Anal. Appl., 304 (2005): 74--86.

\item[{[19]}] R. Ma, H. Gao, Y. Lu, Global structure of radial positive solutions for a prescribed mean curvature problem in a ball,
J. Funct. Anal., 270(7) (2016): 2430--2455.

\item[{[20]}] W. Walter, Ordinary Differential Equations, Springer, New York, 1998.

\end{list}

\end{document}